\documentclass[11pt, reqno]{amsart}

\usepackage{amsthm,amssymb,amstext,amscd,amsfonts,amsbsy,amsrefs,amsxtra,latexsym,amsmath,xcolor,mathrsfs,fancybox,upgreek, soul,url}
\usepackage[english]{babel}
\usepackage[all,cmtip]{xy}
\usepackage{hyperref}
\hypersetup{colorlinks=true,linkcolor=blue,citecolor=magenta}
\usepackage[latin1]{inputenc}
\usepackage{cancel}
\usepackage{comment}
\usepackage{mdframed}
\allowdisplaybreaks

\usepackage{mathtools}
\usepackage{enumerate}
\usepackage{thmtools}
\usepackage{thm-restate}
\usepackage{chngcntr}
\usepackage{enumitem}
\usepackage{etoolbox}
\usepackage{todonotes}

\usepackage{tikz}
\usepackage{verbatim}
\usetikzlibrary{quotes,angles}

\DeclarePairedDelimiter\abs{\lvert}{\rvert}
\DeclarePairedDelimiter\norm{\lVert}{\rVert}

\makeatletter
\let\oldabs\abs
\def\abs{\@ifstar{\oldabs}{\oldabs*}}
\let\oldnorm\norm
\def\norm{\@ifstar{\oldnorm}{\oldnorm*}}
\makeatother

\makeatletter
\glb@settings 
\makeatother

\newtheorem{theorem}{Theorem}

\newtheorem{proposition}[theorem]{Proposition}

\theoremstyle{definition}

\theoremstyle{remark}

\newtheorem*{remark}{Remark}

\numberwithin{theorem}{section}
\numberwithin{proposition}{section}
\numberwithin{lemma}{section}
\numberwithin{corollary}{section}
\numberwithin{equation}{section}
\numberwithin{conjecture}{section}

\setlist[enumerate,1]{before=}
\AfterEndEnvironment{enumerate}{}

\newcommand{\N}{\mathbb{N}}

\usepackage{bm}

\def\lp{\left(}
\def\rp{\right)}

\author{Kevin Gomez}
\address{Department of Mathematics, 1420 Stevenson Center, Vanderbilt University, Nashville, TN 37240}
\email{kevin.j.gomez@vanderbilt.edu}
\author{Joshua Males}
\address{450 Machray Hall, Department of Mathematics, University of Manitoba, Winnipeg,
	Canada}
\email{joshua.males@umanitoba.ca}
\author{Larry Rolen}
\address{Department of Mathematics, 1420 Stevenson Center, Vanderbilt University, Nashville, TN 37240}
\email{larry.rolen@vanderbilt.edu}

\begin{document}

\title[The second shifted difference of partitions and its applications]{The second shifted difference of partitions and its applications}

\subjclass[2020]{11P82}

\keywords{Partitions, convexity, asymptotic expansions}

\thanks{The research of the second author conducted for this paper is supported by the Pacific Institute for the Mathematical Sciences (PIMS). The research and findings may not reflect those of the Institute. This work was supported by a grant from the Simons Foundation (853830, LR). The third author is also grateful for support from a 2021-2023 Dean's Faculty Fellowship from Vanderbilt University and to the Max Planck Institute for Mathematics in Bonn for its hospitality and financial support.}

\begin{abstract}
A number of recent papers have estimated ratios of the partition function $p(n-j)/p(n)$, which appears in many applications. Here, we prove an  easy-to-use effective bound on these ratios. Using this, we then study second shifted difference of partitions, $f(j,n) \coloneqq p(n) -2p(n-j) +p(n-2j)$, and give another easy-to-use estimate of $f(j,n)$. As applications of these, we prove a shifted convexity property of $p(n)$, as well as giving new estimates of the $k$-rank partition function $N_k(m,n)$ and non-$k$-ary partitions along with their differences.
\end{abstract}

\maketitle

\section{Introduction and statement of results}
The study of the values of the partition function $p(n)$, which counts the number of partitions of a positive integer $n$, has a long history. Recall that a partition $\lambda$ of $n$ is a non-increasing list $(\lambda_1, \lambda_2, \dots,\lambda_s)$ such that $\sum_{j=1}^{s} \lambda_j = n$. In their famed collaboration a century ago, Hardy and Ramanujan \cite{HardyRamanujan} proved that as $n\to \infty$ we have the asymptotic formula
\begin{align}\label{p asymp}
p(n) \sim	\frac{1}{4\sqrt{3}n} e^{\pi \sqrt{\frac{2n}{3}}}.
\end{align}
Their proof gave birth to the Circle Method, which is a highly important tool used throughout analytic number theory today. 
Following their discovery, Rademacher \cite{RademacherExact} improved Hardy and Ramanujan's application of the Circle Method to prove an exact formula for $p(n)$. Over the past 100 years, there have been a plethora of investigations into estimates and asymptotics for partitions and their extensions in the literature. 

Here, we will study differences of partition values in detail.
To this end, let $\Delta$ be the backward difference operator defined on sequences $f(n)$ by
\begin{align*}
	\Delta(f(n)) \coloneqq f(n) - f(n-1),
\end{align*}
and its recursive counterpart
\begin{align*}
	\Delta^k (f(n)) \coloneqq \Delta \left( \Delta^{k-1} (f(n)) \right).
\end{align*}
One of the simplest properties of $p(n)$ to prove is that it is convex for $n \geq 2$ (see e.g. \cite{Gupta}), i.e.
\begin{align*}
	p(n) + p(n-2) \geq 2p(n-1).
\end{align*}
Recast using the operator $\Delta$ this is the same as proving that
\begin{align*}
	\Delta^2(p(n)) \geq 0
\end{align*}
for all $n \geq 2$. Gupta \cite{Gupta} also investigated higher powers of $\Delta$ applied to $p$, proving that there exist constants $n_r$ for all $r >0$ such that  $\Delta^r(p(n)) \geq 0$ for all $n \geq n_r$. Moreover, Odlyzko \cite{Od} considered a further conjecture of Gupta, proving that for each $r$ there is a fixed $n_0(r)$ such that for all $n < n_0(r)$ we have that $(-1)^n\Delta^r(p(n)) > 0$ and for all $n \geq n_0(r)$ we have $\Delta^r(p(n)) \geq 0$, as well as giving a beautiful philosophical discussion of why this phenomenon arises. Similar differences of objects related to $p(n)$ and its extensions have been studied by many other authors in the literature, see e.g. \cite{Canfield,Chen,Knessl2,Knessl1,Merca} among many others.

We initiate the investigation of what we call $j$-shifted differences, defined for $1\leq j < n$ on sequences $f(n)$ by
\begin{align*}
	\Delta_j (f(n)) \coloneqq f(n) - f(n-j).
\end{align*}
In analogy to Gupta, it is clear using \eqref{p asymp} that there exist constants $n_{j}$ such that for all $n \geq n_j$ one has that $\Delta_j^2(p(n))\geq 0$. Let $N \coloneqq n-\frac{1}{24}$. Our methods rely on a careful study of the value of the function 
\begin{align*}
	f(j,n) \coloneqq \Delta_j^2(p(n)) = p(n) - 2p(n-j) + p(n-2j),
\end{align*}
and in Theorem \ref{thm:fjn} we prove a precise estimation of $f(j,n)$ with $j \leq \frac{\sqrt{N}}{4}$, in particular providing a strict error term allowing us to closely control the precision of the formula by taking $N$ large enough. In doing so, we provide an easy-to-use estimate for the ratio of partition numbers. Throughout, we use the notation $f(x)= O_\leq(g(x))$ to mean that $\lvert f(x) \rvert \leq g(x)$ for $x$ in the appropriate domain.
\begin{theorem}\label{Thm: main p(n-j)/p(n)}
	Let $n \geq 14$ and $j < \frac{\sqrt{N}}{2}$. Then
	\begin{align*}
		\frac{p(n-j)}{p(n)} = e^{-\frac{\pi j}{\sqrt{6N}}} \left(1+\frac jN-\frac{\pi j^2}{4\sqrt{6}N}-\frac{\sqrt{3}}{\sqrt{2\pi}\sqrt{N}} + O_\leq\left(  \frac{2.71}{N}\right) \right)  \left( 1+ \frac{\sqrt{3}}{\sqrt{2N}\pi} + O_\leq \left( \frac{1350}{N} \right) \right).
	\end{align*}
\end{theorem}

Theorems of a similar flavour to Theorem \ref{Thm: main p(n-j)/p(n)} are abundant in the literature. Lehmer \cite{L1,L2} used Rademacher's exact formula for $p(n)$ \cite{RademacherExact} to provide bounds on the value of $p(n)$ that have seen many applications. More recently, estimates for the ratio of partition values have played a prominent role in proving that the associated Jensen polynomial is eventually hyperbolic \cite{GORZ,LW}, a problem intricately linked with variants of the Riemann hypothesis.

 Theorem~\ref{Thm: main p(n-j)/p(n)} thus applies to many interesting situations. In the remainder of the introduction, we will highlight a few of particular interest.
Our first main result using Theorem \ref{thm:fjn} gives an explicit formula for $n_j$ for ranges of $j$.
\begin{theorem}\label{Thm: main Delta_j>0}
	Let $n \geq 2$ and $j \leq\frac{\sqrt{N}}{4}$. Then we have that 
	\begin{align*}
		\Delta_j^2(p(n)) \geq 0.
	\end{align*}
Equivalently, $p(n)$ satisfies the extended convexity result
	\begin{align*}
		p(n) + p(n-2j) \geq 2p(n-j).
	\end{align*}
\end{theorem}
\begin{remark}
	The methods here should extend to finding formula for $n_{r,j}$ such that for all $n \geq n_{r,j}$ one has $\Delta_j^r(p(n)) \geq 0$, however this would quickly become very lengthy and so we do not pursue this here. Moreover, the referee has kindly pointed out more elementary methods to proving Theorem \ref{Thm: main Delta_j>0} which we elucidate at the end of the paper.
\end{remark}

Our results also apply outside of proving new properties of the partition function itself. We consider the $k$-rank function $N_k(m,n)$ which counts the number of partitions of $n$ into at least $(k-1)$ successive Durfee squares with $k$-rank equal to $m$ \cite{Garvan}. When $k =1$ we recover the number of partitions of $n$ whose Andrews--Garvan crank equals $m$, and when $k=2$ we recover Dyson's partition rank function. Then for $m> \frac{n}{2}$ we have that (see e.g.\@ \cite[page 6]{LZ})
\begin{align*}
	N_k(m,n) = p(n-k-m+1) - p(n-k-m) ,\qquad N_k(m,n) - N_k(m+1,n) = f(1,n-k-m).
\end{align*}
We give precise formulae both $N_k(m,n)$ and for the differences of $k$-ranks in certain ranges of $m$ in the following theorems, improving on  \cite[Theorem 1.4]{LZ} in this range. The proof follows from a direct application of Theorem \ref{Thm: main p(n-j)/p(n)}.

\begin{theorem}
	Let $m > \frac{n}{2}$ and $\ell \coloneqq n-k-m+\frac{23}{24} >16$. Then we have that
	\begin{align*}
		\frac{N_k(m,n)}{p(n-k-m+1)} = 1-e^{-\frac{\pi}{\sqrt{6\ell}}} \left(1- \frac{\sqrt{3}}{\sqrt{2\pi \ell}} + O_\leq\left(  \frac{4.04}{\ell}\right) \right)  \left( 1+ \frac{\sqrt{3}}{\sqrt{2\ell}\pi} + O_\leq \left( \frac{1350}{\ell} \right) \right).
	\end{align*}
\end{theorem}

We then turn to obtaining a precise estimate for the differences of $k$-ranks, with the proof following from a direct application of Theorem \ref{thm:fjn}.

\begin{theorem}\label{Thm: k-ranks precise}
	Let $m> \frac{n}{2}$ and $\ell  =  n-k-m+\frac{23}{24} > 16$. Then we have that
	\begin{align*}
		\frac{	N_k(m,n) - N_k(m+1,n)}{p(n-k-m+1)} = 1&+ e^{-\frac{\sqrt{2}\pi}{\sqrt{3\ell}}} \left( 1+   \left( \frac{\sqrt{3}}{\sqrt{2}\pi} - \frac{\sqrt{3}}{\sqrt{2 \pi}} \right) \frac{1}{\sqrt{\ell}} + O_\leq \left( \frac{2079}{\ell} \right)\right)\\
		& - e^{-\frac{\pi}{\sqrt{6\ell}}} \left(  2 + \left( \frac{2\sqrt{3}}{\sqrt{2}\pi} - \frac{2\sqrt{3}}{\sqrt{2 \pi}} \right) \frac{1}{\sqrt{\ell}} +  O_\leq \left( \frac{3929}{\ell} \right) \right).
	\end{align*}
\end{theorem}

As a direct implication, we recover positivity of the differences of $k$-ranks in these cases, as in \cite[Corollary 1.5]{LZ}.

Our final application is to so-called non-$k$-ary partitions ($k \in \N$), recently defined by Schneider \cite{Sch} as partitions of $n$ with no parts equal to $k$\footnote{While \cite{Sch} uses the terminology ``$k$-nuclear", Schneider has recommended the authors use the term non-$k$-ary based on advice of G. Andrews to better fit the case of $k=1$, classically called the non-unitary partitions.}. Letting $\nu_k(n)$ be the number of non-$k$-ary partitions of $n$, it is clear that $\nu_k(n) = p(n) -p(n-k)$. By Theorem \ref{Thm: main p(n-j)/p(n)} we immediately obtain an effective estimate for the ratio $\nu_k(n)/p(n)$, improving on \cite[Theorem 1]{akande}. We also have 
\begin{align*}
	\nu_k(n) - \nu_k(n-k) =  f(k,n),
\end{align*}
and so we also obtain precise estimates for differences of non-$k$-ary partitions using Theorem \ref{thm:fjn} for $k < \frac{\sqrt{N}}{4}$, with a direct implication being the following theorem.

\begin{theorem}
	For $n \geq 2$ and $k \leq \frac{\sqrt{N}}{4}$ we have
	\begin{align*}
		\nu_k(n) - \nu_k(n-k) > 0.
	\end{align*}
\end{theorem}

\section*{Acknowledgments}
The authors are very grateful to Kathrin Bringmann who shared many calculations and insightful comments invaluable for this paper. We are also grateful to the referee for pointing out the paper of Odlyzko and a more elementary proof of Theorem 1.2 which we have discussed at the end of the paper.
\section{The proofs}
In this section we prove the main results of the paper. We begin by proving a technical estimate of the value of $p(n-j)$, utilising Rademacher's exact formula for the partition function.

\begin{proposition}\label{prop}
	Let $N \coloneqq n-\frac{1}{24}$ and $j \in \N_0$. Then
	\begin{multline*}
			p(n-j) = \frac{e^{\pi\sqrt{\frac{2(N-j)}{3}}}}{4\sqrt{3}(N-j)}\\
		\times \left(1-\frac{\sqrt{3}}{\sqrt{2}\pi\sqrt{N-j}} + O_\le\left(\frac{2\pi^2(N-j)e^{-\pi\sqrt{\frac{2(N-j)}{3}}}}{3} + 2^3 3^{-\frac12} \pi (N-j)^{\frac12} e^{-\frac\pi2\sqrt{\frac{N-j}{2}}} \right)\right).
	\end{multline*}

\end{proposition}

\begin{proof}
We first recall the following result from \cite[Theorem 1.1]{IJT} with $\alpha = 1$, which is simply Rademacher's exact formula for the partition function \cite{RademacherExact},
\begin{equation}\label{Eqn: Rademacher}
	p(n) = \frac{\pi}{2^{\frac 54} 3^{\frac 34} N^{\frac 34}}\sum\limits_{k=1}^\infty \frac{A_k(n)}{k}I_{\frac 32}\left( \frac{\pi}{k} \sqrt{\frac{2N}{3}} \right),
\end{equation}
where $I_{\nu}$ is the usual $I$-Bessel function and
\begin{equation*}
	A_k(n):= \sum\limits_{\substack{0 \le h < k \\ \gcd(h,k) = 1}}e^{\pi i s(h,k) - \frac{2\pi i n h}{k} }
\end{equation*}
is a Kloosterman sum with $s(h,k)$ the usual Dedekind sum. By page 172 of \cite{Watson}, we have the following representation of the $I$-Bessel function
\begin{equation}\label{Eqn: Bessel integral}
	I_{\frac32}(x) = \frac{x^{\frac32}}{2\sqrt{2\pi}} \int\limits_{-1}^1 \left(1-t^2\right)e^{xt}dt.
\end{equation}
We now bound the integrand for $-1\le t\le0$ by $1$ and find that 
\begin{equation}\label{Eqn: Bessel error}
	\int\limits_{-1}^0 \left(1-t^2\right)e^{xt}dt = O_\le(1).
\end{equation}
Here, the notation $f(x)=O_\le(g(x))$ means that $|f(x)|\leq g(x)$, that is, that there is no implied constant in the big-Oh estimate.

Next we compute the integral for $0 \le t \le 1$. To do so, we make the change of variables $u=1-t$ to find that 
\begin{align*}
	\int\limits_0^1 \left(1-t^2\right)e^{xt}dt &= e^x \left(2\int\limits_0^1 u^2e^{-xu}\frac{du}{u} - \int\limits_0^1 u^3e^{-xu}\frac{du}{u}\right).
\end{align*}
Under the change of variables $w =ux$, it is easy to show that this is equal to
\begin{align*}
 \frac{2e^x}{x^2}\left(1-\frac1x-\Gamma(2,x)+\frac{\Gamma(3,x)}{2x}\right),
\end{align*}
where $\Gamma(a,b)$ is the usual incomplete $\Gamma$-function. Since $\Gamma(2,x) = (x+1)e^{-x}$ and $\Gamma(3,x) = (x^2+2x+2)e^{-x}$, we find that 
\begin{equation}\label{Eqn: Bessel main term}
	\int\limits_0^1 \left(1-t^2\right)e^{xt} dt = \frac{2e^x}{x^2} \left(1-\frac1x-\frac{x+1}{e^x}+\frac{x^2+2x+2}{2xe^x}\right).
\end{equation}
Plugging \eqref{Eqn: Bessel error} and \eqref{Eqn: Bessel main term} into \eqref{Eqn: Bessel integral}, we obtain
\begin{align*}
	I_{\frac32}(x) &= \frac{x^{\frac32}}{2\sqrt{2\pi}} \left(\frac{2e^x}{x^2} \left(1-\frac1x-\frac{x+1}{e^x}+\frac{x^2+2x+2}{2xe^x}\right) + O_\le(1)\right)\\
	&= \frac{e^x}{\sqrt{2\pi x}} \left(1-\frac1x+\left(\frac1x-\frac x2\right) e^{-x}+O_\le\left(\frac{x^2}{2e^x}\right)\right).
\end{align*}
Noting that for $x\ge1$, one has $|\frac1x-\frac x2|\le\frac{x^2}{2}$ we have 
\begin{equation*}
	I_{\frac{3}{2}}(x) = \frac{e^x}{\sqrt{2\pi x}}\left(1 - \frac{1}{x} + O_\le \left(x^2 e^{-x}\right)\right) 
\end{equation*}
for $x \ge 1$. In particular,
\begin{equation}\label{Eqn: main Bessel}
	I_{\frac32}\left(\pi\sqrt{\frac{2N}{3}}\right) = \frac{3^{\frac14}e^{\pi \sqrt{\frac{2N}{3}}}}{2^{\frac34}\pi N^{\frac14}} \left(1-\frac{\sqrt{3}}{\sqrt{2}\pi} \frac{1}{\sqrt{N}} + O_\le\left(\frac{2\pi^2Ne^{-\pi\sqrt{\frac{2N}{3}}}}{3}\right)\right)
\end{equation}
for $n\ge1$. This corresponds to the term $k=1$ in the sum in \eqref{Eqn: Rademacher}, and we need to bound the remaining terms of the sum. Note that $|A_k(n)| \le k$, so we may bound the remaining terms in the sum by 
\begin{equation*}
	\left|\sum\limits_{k=2}^\infty \frac{A_k(n)}{k} I_{\frac 32}\left( \frac{\pi}{k}\sqrt{\frac{2N}{3}}\right) \right| \le \sum\limits_{k=2}^\infty I_{\frac 32}\left( \frac{\pi}{k}\sqrt{\frac{2N}{3}}\right).
\end{equation*}
We now emulate (3.20) of \cite{BKRT}. By (3.18) of \cite{BKRT} we have
\[
	\sum_{2\le k\le\lfloor X\rfloor} I_{\frac32} \left(\frac Xk\right) \le 2\sqrt{\frac X\pi} e^{\frac X2}.
\]
The remaining terms are
\begin{align*}
	\sum_{k \geq \lfloor X \rfloor +1} I_{\frac{3}{2}} \left(\frac{X}{k}\right) \leq  \frac{X^{\frac{3}{2}}}{\Gamma \left(\frac{5}{2}\right) \sqrt{2}} \sum_{k \geq \lfloor X \rfloor + 1} \frac{1}{k^{\frac{3}{2}}},
\end{align*}
where we use Lemma 2.2 (3) of \cite{BKRT}. We thus have the bound
\begin{align*}
  \sum_{k \geq 2} I_{\frac{3}{2}} \left(\frac{X}{k}\right) \leq 2 \sqrt{\frac{X}{\pi}} e^{\frac{X}{2}} +	\frac{2X^{\frac{3}{2}}}{\Gamma\left(\frac{5}{2}\right) \sqrt{2}}.
\end{align*}
It remains to bound the final sum with basic calculus by $2\sqrt{\frac X\pi}e^{\frac X2}$, yielding
\begin{align*}
	 \sum_{k \geq 2} I_{\frac{3}{2}}  \left(\frac{X}{k}\right)  \leq 4 \sqrt{\frac{X}{\pi}} e^{\frac{X}{2}}.
\end{align*}
 In our application, this yields
\begin{equation}\label{Eqn: error Bessels}
	\sum_{k=2}^\infty I_{\frac32}\left(\frac\pi k\sqrt{\frac{2N}{3}}\right) \le 4\sqrt[3]{\frac{2N}{3}}e^{\frac\pi2\sqrt{\frac{2N}{3}}} = \frac{2^{\frac94}N^{\frac14}}{3^{\frac14}} e^{\frac\pi2\sqrt{\frac{2N}{3}}}.
\end{equation}
Combining \eqref{Eqn: main Bessel} and \eqref{Eqn: error Bessels} in \eqref{Eqn: Rademacher}, we find that 
\begin{align*}
	p(n) &= \frac{\pi}{2^{\frac54}3^{\frac34}N^{\frac34}} \left(\frac{3^{\frac14}e^{\pi\sqrt{\frac{2N}{3}}}}{2^{\frac34}\pi N^{\frac14}} \left(1-\frac{\sqrt{3}}{\sqrt{2}\pi\sqrt{N}} + O_\le\left(\frac{2\pi^2Ne^{-\pi\sqrt{\frac{2N}{3}}}}{3}\right)\right) + O_\le\left(\frac{2^{\frac94}N^{\frac14}}{3^{\frac14}} e^{\frac\pi2\sqrt{\frac{2N}{3}}}\right)\right)\\
	&= \frac{e^{\pi\sqrt{\frac{2N}{3}}}}{4\sqrt{3}N} \left(1-\frac{\sqrt{3}}{\sqrt{2}\pi\sqrt{N}} + O_\le\left(\frac{2\pi^2Ne^{-\pi\sqrt{\frac{2N}{3}}}}{3} + 2^33^{-\frac12} \pi N^{\frac12} e^{-\frac\pi2\sqrt{\frac N2}}\right)\right).
\end{align*}
Note that $N=N(n)$ is implicitly a function of $n$ and $N(n-j) = N(n) - j$ and so the claim follows.
\end{proof}

Next we want to estimate the functions $f(j,n)$. The first step in the proof is to obtain estimates of $\frac{p(n-j)}{p(n)}$ analogous to those of \cite{BKRT}, proving Theorem \ref{Thm: main p(n-j)/p(n)} en-route.

\begin{theorem}\label{thm:fjn}
	Let $j<\frac{\sqrt{N}}{4}$ and $n\geq14$. Then
	\begin{align*}
		\frac{f(j,n)}{p(n)} =  1 &+ e^{-\frac{\sqrt{2}\pi j}{\sqrt{3N}}} \left( 1+   \left( \frac{\sqrt{3}}{\sqrt{2}\pi} - \frac{\sqrt{3}}{\sqrt{2 \pi}} \right) \frac{1}{\sqrt{N}} + \frac{2j}{N} - \frac{\pi j^2}{\sqrt{6} N^{\frac{3}{2}}} + O_\leq \left( \frac{2075}{N} \right) \right)\\
		& - e^{-\frac{\pi j}{\sqrt{6N}}} \left(  2 + \left( \frac{2\sqrt{3}}{\sqrt{2}\pi} - \frac{2\sqrt{3}}{\sqrt{2 \pi}} \right) \frac{1}{\sqrt{N}} + \frac{2j}{N} - \frac{\pi j^2}{2\sqrt{6}N^{\frac{3}{2}}} + O_\leq \left( \frac{3926}{N} \right) \right).
	\end{align*}

\rm
\end{theorem}

\begin{proof}
	 Recall that
	\[
		\frac{f(j,n)}{p(n)} = 1-2\frac{p(n-j)}{p(n)}+\frac{p(n-2j)}{p(n)}.
	\]
	We first bound $\frac{1}{p(n)}$. By Proposition \ref{prop} we have
	\[
		p(n) = \frac{e^{\pi\sqrt{\frac{2N}{3}}}}{4\sqrt{3}N} \left(1-\frac{\sqrt{3}}{\sqrt{2}\pi\sqrt{N}}+g(N)\right),
	\]
	where
	\[
		|g(N)| \le \frac{2\pi^2Ne^{-\pi\sqrt{\frac{2N}{3}}}}{3} + 2^3\cdot3^{-\frac12}\pi N^\frac12 e^{-\frac\pi2\sqrt{\frac N2}} =: h(N).
	\]
	Now we want to approximate
	\[
		\frac{1}{1-\left(\frac{\sqrt{3}}{\sqrt{2}\pi\sqrt{N}}-g(N)\right)}.
	\]
	We have for $N$
	\[
		\left|\frac{\sqrt{3}}{\sqrt{2}\pi\sqrt{N}}-g(N)\right| < \frac{\sqrt{3}}{\sqrt{2}\pi \sqrt{N}} + |g(N)| \le \frac{\sqrt{3}}{\sqrt{2}\pi \sqrt{N}} + h(N) < 0.99,
	\]
	which can be seen by taking Taylor series. Now we claim that for $0<|z|<0.99$
	\[
		\frac{1}{1-z} = 1+z+O_\leq\left(100|z|^2\right).
	\]
	To see this, we bound
	\[
		\left|\frac{1}{1-z}-1-z\right| = \left|\frac{1-(1+z)(1-z)}{1-z}\right| = \frac{|z|^2}{|1-z|} \le \frac{|z|^2}{1-|z|} < \frac{1}{0.01}|z|^2 = 100|z|^2.
	\]
	Thus
	\begin{align*}
		\frac{1}{p(n)} &= \frac{4\sqrt{3}Ne^{-\pi\sqrt{\frac{2N}{3}}}}{1-\frac{\sqrt{3}}{\sqrt{2}\pi\sqrt{N}}+g(N)}\\
		&= 4\sqrt{3}Ne^{-\pi\sqrt{\frac{2N}{3}}}\left(1+\frac{\sqrt{3}}{\sqrt{2}\pi\sqrt{N}}-g(N) + O_\le\left(100\left|\frac{\sqrt{3}}{\sqrt{2}\pi\sqrt{N}}-g(N)\right|^2\right)\right)\\
		&= 4\sqrt{3}Ne^{-\pi\sqrt{\frac{2N}{3}}}\left(1+\frac{\sqrt{3}}{\sqrt{2}\pi\sqrt{N}} + O_\le\left(|g(N)|+100\left|\frac{\sqrt{3}}{\sqrt{2}\pi\sqrt{N}}-g(N)\right|^2\right)\right).
	\end{align*}
	Now the error may be bound against
	\[
		h(N)+100\left(\frac{\sqrt{3}}{\sqrt{2}\pi\sqrt{N}}+h(N)\right)^2 = O_\le\left(\frac{1350}{N}\right)
	\]
	again by basic calculus.
	Thus 
	$$
	\frac{1}{p(n)}=4\sqrt{3}Ne^{-\pi \sqrt{\frac{2N}{3}}} \lp 1+\frac{\sqrt{3}}{\sqrt{2}\pi \sqrt{N}}+O_{\leq}\lp \frac{1350}{N} \rp \rp.
	$$
	
	Consider $p(n-J), J\in \{j,2j\}.$ Then by Proposition \ref{prop}, we have
	\begin{equation}\label{PNJ}
	p(N-J)=\frac{e^{\pi \sqrt{\frac23}\sqrt{N-J}}}{4\sqrt{3}(N-J)}
	\lp 1-\frac{\sqrt{3}}{\sqrt{2}\pi \sqrt{N-J}}+O_\leq \lp h(N-J) \rp \rp.
	\end{equation}

	First note that $j<\frac{\sqrt{N}}{4}$ implies that (note that $N\ge14-\frac{1}{24}$)
	\[
		N-J \ge N-2j \ge N-\frac{\sqrt{N}}{2} \ge 12.
	\]
	Moreover since $1\le j\le\frac{\sqrt{N}}{4}$
	\[
		\frac JN \le \frac{2j}{N} \le \frac{\sqrt{N}}{2N} = \frac{1}{2\sqrt{N}} \le \frac{1}{2\sqrt{14-\frac{1}{24}}} < 0.2.
	\]
	We now first approximate the exponential in (\ref{PNJ}). We claim that
	$$
	e^{\pi \sqrt{\frac23}\sqrt{N-J}}
	=e^{\pi \sqrt{\frac23}\sqrt{N}-\frac{\pi J}{\sqrt{6N}}}
	\lp 1-\frac{\pi J^2}{4\sqrt{6}N^{\frac32}}+O_\leq \lp  \frac{0.1}{N} \rp \rp.
	$$
	To see this define
	$$
	f(x):=\sqrt{1-x}-1+\frac x2 +\frac{x^2}{8}.
	$$
	It is straightforward to show that for $0\leq x <0.2$ we have
	$$
	|f(x)|\leq 0.1x^3.
	$$
	We now use this to write
	$$
	\sqrt{N-J}=\sqrt{N}\sqrt{1-\frac JN}=\sqrt{N}
	\lp f\lp \frac JN\rp+1- \frac{J}{2N} -\frac18 \lp\frac JN\rp^2 \rp
	$$
	and obtain
	$$
	e^{\pi \sqrt{\frac23} \sqrt{N-J}-\pi \sqrt{\frac23}\sqrt{N}+\frac{\pi J}{\sqrt{6N}}}
	=e^{-\pi \sqrt{\frac23}\sqrt{N}\lp-f\lp \frac JN\rp+ \frac{J^2}{8N^2}\rp}.
	$$
	Note that
	$$
	-f(x)+\frac{x^2}{8} =1-\frac x2-\sqrt{1-x}=\frac{x^2}{8}+\ldots>0.
	$$
	Also note that for $x>0$
	$$
	e^{-x}-1+x=\sum_{n\geq 2} \frac{(-1)^nx^n}{n!}\leq \frac{x^2}{2}
	$$
	by Leibnitz's criterium.
	Thus,
	\begin{align*}
		&e^{-\pi\sqrt{\frac23}\sqrt{N}\lp-f\lp \frac JN\rp+ \frac{J^2}{8N^2}\rp}
		\\
		&=1-\pi\sqrt{\frac23}\sqrt{N}\lp-f\lp \frac JN\rp+ \frac{J^2}{8N^2}\rp
		+O_\leq \lp \frac12 \lp \pi \sqrt{\frac23} \sqrt{N} \lp-f \lp \frac JN\rp+\frac{J^2}{8N^2}\rp\rp^2\rp
		\\
		&=1-\frac{\pi J^2}{\sqrt{6} N^{\frac32}}+O_\leq \lp \pi \sqrt{\frac23} \sqrt{N} \left| f \lp \frac JN\rp	\right|
		+\frac12 \lp \pi\sqrt{\frac23}\sqrt{N}\lp\left|f\lp \frac JN\rp\right|+ \frac{J^2}{8N^2}\rp\rp^2
		\rp.		
	\end{align*}
	We now bound the error against
	\begin{align*}
		\pi\sqrt{\frac23} \sqrt{N}&\cdot 0.1 \lp \frac JN\rp^3
		+\frac{\pi^2}{3}N 
		\lp 0.1 \lp \frac JN\rp^3+ \frac{J^2}{8N^2} \rp^2
		\\
		&\leq \pi\sqrt{\frac23} 0.1 \frac{1}{2^3 N}+\frac{\pi^2}{3}N 
		\lp 0.1 \lp \frac{1}{2\sqrt{N}}\rp^3+\frac18 \lp \frac{1}{2\sqrt{N}} \rp^2 \rp^2
		\\
		&=\frac{0.1\pi}{4\sqrt{6}} \frac1N + \frac{\pi^2}{24} \frac1N \lp \frac{0.1}{\sqrt{N}}+\frac14 \rp^2
		=\lp \frac{0.1\pi}{4\sqrt{6}}+\frac{\pi^2}{24}\lp \frac{0.1}{\sqrt{N}}+\frac14\rp^2\rp
		\frac1N 
		\leq \frac{0.1}{N},
	\end{align*}
where in the final inequality we use that $N\ge 14-\frac{1}{24}$. This gives the claim.
	
	Next we claim that for $x<0.2$
	\[
	\frac{1}{1-x} = 1+x+O_\le\left(1.25 x^2\right).
	\]
	To see this, we bound
	\[
	\left|\frac{1}{1-x}-1-x\right| = \frac{x^2}{|1-x|} \le \frac{x^2}{1-|x|} \le \frac{x^2}{0.8}.
	\]
	We use this for
	\[
	\frac{1}{N-J} = \frac1N \frac{1}{1-\frac JN} = \frac1N\left(1+\frac{J}{N}+O_\le\left(1.25\left(\frac JN\right)^2\right)\right)=\frac1N \lp 1+\frac{J}{N}+O_\leq \lp \frac{0.4}{N} \rp \rp.
	\]
	
		Next we use that for $0\le x<0.2$ 
	\[
	\frac{1}{\sqrt{1-x}}-1 \le 0.6x.
	\]
	Thus
	\begin{align}\label{expand} \nonumber
	\frac{1}{\sqrt{N-J}} &= \frac{1}{\sqrt{N}}\frac{1}{\sqrt{1-\frac JN}} = \frac{1}{\sqrt{N}}\left(1+O_\le\left(0.6\frac JN\right)\right) = \frac{1}{\sqrt{N}}+O_\le\left(0.6\frac{J}{N^\frac 32}\right)
	\\
	&=\frac{1}{\sqrt{N}}+O_\leq \lp \frac{0.3}{N}\rp.
	\end{align}
	
	Finally, by basic calculus we find the bound 
	\[
	\frac{2\pi^2xe^{-\pi\sqrt{\frac{2x}{3}}}}{3} + 2^33^{-\frac12}\pi x^\frac12 e^{-\frac\pi2\sqrt{\frac x2}} \le 15x^\frac12 e^{-\frac\pi2\sqrt{\frac x2}}
	\]
	for $x\ge12$. 
	Thus
	$$
	|f(N-J)|\leq 15(N-J)^{\frac12} e^{-\frac{\pi}{2\sqrt{2}}\sqrt{N-J}}.
	$$
	Now note that
	\[
	x^\frac12e^{-\frac\pi2\sqrt{\frac x2}}
	\]
	is decreasing for $x\ge1$. We then use the bound
	\[
	N-J \ge N-\frac{\sqrt{N}}{2}
	\]
	and thus
	\[
	(N-J)^\frac12 e^{-\frac\pi2\sqrt{\frac{N-J}{2}}} \le \left(N-\frac{\sqrt{N}}{2}\right)^\frac12 e^{-\frac\pi2\sqrt{\frac{N-\frac{\sqrt{N}}{2}}{2}}} = O_\le\left(\frac{1.1}{N}\right),
	\]
	Thus
	\begin{multline*}
		p(n-J) = \frac{1}{4\sqrt{3}}e^{\pi\sqrt{\frac{2N}{3}}-\frac{\pi J}{\sqrt{6N}}} \left(1-\frac{\pi J^2}{4\sqrt{6}N^\frac32}
		+O_\le\left(\frac{0.1}{N}\right)\right) \frac1N\left(1+\frac{J}{N}+O_\le\left(\frac{0.4}{N} \right)\right)\\
		\times \left(1-\frac{\sqrt{3}}{\sqrt{2\pi}}\left(\frac{1}{\sqrt{N}} + O_\le\left(\frac{0.3}{N}\right)\right) + O_\le\left(\frac{1.1}{N}\right)\right).
	\end{multline*}
	We combine
	\begin{multline*}
		\left(1-\frac{\pi J^2}{4\sqrt{6}N^\frac32}+O_\le\left(\frac{0.1}{N}\right)\right) \left(1+\frac JN+O_\le\left(\frac{0.4}{N}\right)\right)\\
		= 1+\frac JN-\frac{\pi J^2}{4\sqrt{6}N^\frac32} + O_\le\left(\frac{0.4}{N}+\frac{\pi J^3}{4\sqrt{6}N^\frac52}+\frac{0.4\pi J^2}{4\sqrt{6}N^\frac52}+\frac{0.1}{N}+0.1\frac{J}{N^2}+0.1\cdot \frac{0.4}{N^2}\right).
	\end{multline*}
	The error may be bounded against $\frac{0.56}{N}$. 
	
	Next, we estimate
	\[
		1-\frac{\sqrt{3}}{\sqrt{2\pi}} \left(\frac{1}{\sqrt{N}}+O_\le\left(\frac{0.3}{N}\right)+O_\le\left(\frac{1.1}{N}\right)\hspace{-.1cm}\right) = 1-\frac{\sqrt{3}}{\sqrt{2\pi}\sqrt{N}} + O_\le \left( \frac{1.31}{N} \right).
	\]
	Thus,
	\begin{align*}
		&4\sqrt{3}Ne^{-\pi\sqrt{\frac{2N}{3}}+\frac{\pi J}{\sqrt{6N}}} p(n-J)\\
		&\hspace{1.275cm}= \left(1+\frac JN-\frac{\pi J^2}{4\sqrt{6}N^\frac32}+O_\le\left(\frac{0.56}{N}\right)\right) \left(1-\frac{\sqrt{3}}{\sqrt{2\pi}\sqrt{N}}+O_\le\left(\frac{1.31}{N}\right)\right)\\
		&\hspace{1.275cm}= 1+\frac JN - \frac{\pi J^2}{4\sqrt{6}N^\frac32}-\frac{\sqrt{3}}{\sqrt{2\pi}\sqrt{N}} + O_\le\left( \frac{2.71}{N} \right).
	\end{align*}
Note that this calculation also proves Theorem \ref{Thm: main p(n-j)/p(n)}. Thus, overall, we obtain
	\small\begin{align}\label{eqn final}
		\frac{f(n,j)}{p(n)} =& 1+\frac{1}{p(n)}(p(n-2j)-2p(n-j)) \notag\\
		=& 1+4\sqrt{3}Ne^{-\pi\sqrt{\frac{2N}{3}}} \left(1+\frac{\sqrt{3}}{\sqrt{2}\pi\sqrt{N}}+O_\le\left(\frac{1350}{N}\right)\right) \frac{e^{\pi\sqrt{\frac{2N}{3}}}}{4\sqrt{3}N}  \notag\\
		&\times  \Bigg(e^{-\frac{\sqrt{2}\pi j}{\sqrt{3N}}} \left(1+\frac{2j}{N}-\frac{\pi j^2}{\sqrt{6}N^\frac32}-\frac{\sqrt{3}}{\sqrt{2\pi}\sqrt{N}} +O_\leq\left(\frac{2.71}{N}\right)\right) \notag\\
		& \hspace{5cm}- 2e^{-\frac{\pi j}{\sqrt{6N}}} \left(1+\frac jN-\frac{\pi j^2}{4\sqrt{6}N^{\frac{3}{2}}}-\frac{\sqrt{3}}{\sqrt{2\pi}\sqrt{N}} +O_\leq\left(\frac{2.71}{N}\right) \right)\Bigg).
	\end{align}
\normalsize
Using that $j \leq\frac{\sqrt{N}}{4}$  and $n \geq 14$ a straightforward calculation gives 
\begin{multline*}
	\left(1+\frac{\sqrt{3}}{\sqrt{2}\pi\sqrt{N}}+O_\le\left(\frac{1350}{N}\right) \right) \left(1+\frac{2j}{N}-\frac{\pi j^2}{\sqrt{6}N^{\frac32}}-\frac{\sqrt{3}}{\sqrt{2\pi}\sqrt{N}} 
	+O_\leq\left(\frac{2.71}{N}\right)\right) \\
	=1 + \left( \frac{\sqrt{3}}{\sqrt{2}\pi} - \frac{\sqrt{3}}{\sqrt{2 \pi}} \right) \frac{1}{\sqrt{N}} + \frac{2j}{N} - \frac{\pi j^2}{\sqrt{6} N^{\frac{3}{2}}}
	+ O_\leq \left( \frac{2075}{N} \right).
\end{multline*}
\normalsize 

We turn to the final product in \eqref{eqn final} given by
\begin{align*}
	& 2\left(1+\frac{\sqrt{3}}{\sqrt{2}\pi\sqrt{N}}+O_\le\left(\frac{1350}{N}\right)\right) \left(1+\frac jN-\frac{\pi j^2}{4\sqrt{6}N^{\frac{3}{2}}}-\frac{\sqrt{3}}{\sqrt{2\pi}\sqrt{N}} +O_\leq\left(\frac{2.71}{N}\right) \right).
\end{align*}
Again using $j \leq \frac{\sqrt{N}}{4}$ and $n \geq 14$ it is not hard to show that this is equal to
\begin{align*}
	2 + \left( \frac{2\sqrt{3}}{\sqrt{2}\pi} - \frac{2\sqrt{3}}{\sqrt{2 \pi}} \right) \frac{1}{\sqrt{N}} + \frac{2j}{N} - \frac{\pi j^2}{2\sqrt{6}N^{\frac{3}{2}}} + O_\leq \left( \frac{3926}{N} \right)
\end{align*}
\normalsize
Combining everything together, we obtain the statement of the theorem.
\end{proof}

We end by proving the eventual positivity of the ratio $\frac{f(j,n)}{p(n)}$, which is crucial to the applications in the introduction.

\begin{theorem} \label{thm:app1.2}
	Let $j \leq \frac{\sqrt{N}}{4}$. Then we have that $\frac{f(j,n)}{p(n)} > 0$ for all $n \geq 2$.
\end{theorem}

\begin{proof} 
	Let $X = e^{\frac{\sqrt{2}\pi j}{\sqrt{3N}}} - 2 e^{\frac{\pi j }{\sqrt{6N}}}$. We have by Theorem \ref{thm:fjn} that the result follows if
	\begin{align}\label{eqn: first ineq}
		\frac{f(j,n)}{p(n)} - 1 = &\left( 1+ \left( \frac{\sqrt{3}}{\sqrt{2}\pi} - \frac{\sqrt{3}}{\sqrt{2\pi}} \right)\frac{1}{\sqrt{N}} + \frac{j}{N} - \frac{\pi j^2}{4 \sqrt{6} N^{\frac{3}{2}}} +  \frac{3926}{N} \right) X \notag \\ &+ e^{-\frac{\sqrt{2}\pi j}{\sqrt{3N}}} \left( \frac{j}{N} -  \frac{3 \pi j^2}{4 \sqrt{6} N^{\frac{3}{2}}} \right) > -1
	\end{align}
	for all $n \geq 14$, with a finite computer check taking care of the remaining cases.
	
	We first observe using $j \leq \frac{\sqrt{N}}{4}$ that 
	\begin{align*}
	\frac{3 \pi j^2}{4 \sqrt{6} N^{\frac{3}{2}}} <	\frac{j}{N},
	\end{align*}
	implying that the final term in \eqref{eqn: first ineq} is positive. Thus, the claim follows if
	\begin{align*}
		 \left( 1+ \left( \frac{\sqrt{3}}{\sqrt{2}\pi} - \frac{\sqrt{3}}{\sqrt{2\pi}} \right)\frac{1}{\sqrt{N}} + \frac{j}{N} - \frac{\pi j^2}{4 \sqrt{6} N^{\frac{3}{2}}} +  \frac{3926}{N} \right) X > -1.
	\end{align*}

	Since $-1 < X < 0$ for all $n \geq 4$, this follows if
	
	\begin{align*}
		\left( \frac{\sqrt{3}}{\sqrt{2}\pi} - \frac{\sqrt{3}}{\sqrt{2\pi}} \right)\frac{1}{\sqrt{N}} + \frac{j}{N} - \frac{\pi j^2}{4 \sqrt{6} N^{\frac{3}{2}}} +  \frac{3926}{N} < 0.
	\end{align*}
	It is simple to check that this is always satisfied for $n \geq 14$, and the theorem follows. We note that this also proves Theorem \ref{Thm: main Delta_j>0}.
\end{proof}

We are grateful to the referee for pointing out the following more elementary approaches to proving Theorem \ref{Thm: main Delta_j>0}, and in fact a wider class of inequalities for $p(n)$. Let $V_j(n)$ be the set of partitions of non-$j$-ary partitions. Then for any $\ell \geq 0$ we may construct a map 
\begin{align*}
	\pi \colon V_j(n-\ell) &\to V_j(n) \\
	(\lambda_1,\dots,\lambda_s) &\mapsto (\lambda_1+\ell, \lambda_2,\dots,\lambda_s).
\end{align*}
Since this map is clearly injective, we immediately obtain that
\begin{align*}
	p(n-\ell) - p(n-\ell-j) \leq p(n) - p(n-j)
\end{align*}
for all $\ell \geq 0$. Choosing $\ell = j$ we recover Theorem \ref{Thm: main Delta_j>0}. One may also write this in terms of coefficients of $q$-series as follows. We have that
\begin{align*}
	\Delta^r_j(p(n)) = [q^n] (1-q^j)^r \prod_{k \geq 1} (1-q^k)^{-1}
\end{align*}
as in \cite[equation (2.4)]{Od}. When $r=2$, it is readily checked (using the $q$-binomial theorem) that the $q$-series has non-negative coefficients, giving Theorem \ref{Thm: main Delta_j>0}. However, if one fixes $j$ and asks about the behaviour as $r \to \infty$ it is less clear whether the $q$-series has non-negative coefficients. For $j=1$ this is Gupta's conjecture \cite{Gupta}. Moreover, in \cite{Od} it is shown that $\Delta_1^r(p(n))$ alternates in sign before eventually becoming non-negative. Does a similar phenomenon hold for $\Delta_j^r(p(n))$?

We remark that these more elementary approaches rely on the combinatorial structure and the infinite product representations that occur for $p(n)$. For other objects with similar asymptotic behaviour to $p(n)$, our analytical techniques provide a pathway to similar inequalities where one may not have a combinatorial interpretation or infinite product representation.


\begin{thebibliography}{99}
	

\bibitem{akande} A. Akande, T. Genao, S. Haag, M. Hendon, N. Pulagam, R. Schneider, and A. Sills, {\it Computational study of non-unitary partitions},  \url{https://arxiv.org/abs/2112.03264}, preprint.

	\bibitem{BKRT} K. Bringmann,  B. Kane, L. Rolen, and Z. Tripp, {\it Fractional partitions and conjectures of Chern-Fu-Tang and Heim-Neuhauser}, Trans. Am. Math. Soc., Series B, accepted for publication.
	
	\bibitem{Canfield} R. Canfield, S. Corteel, and P. Hitczenko, \textit{Random partitions with non-negative {$r$}-th differences}, Adv. in Appl. Math., Special issue in honor of Dominique Foata's 65th birthday (Philadelphia, PA, 2000), {\bf 27} (2001), no. 2-3, 298--317.
	
\bibitem{Chen} W. Y. C. Chen,L. X. W. Wang, and G. Y. B. Xie, \textit{Finite differences of the logarithm of the partition function}, Math. Comp., {\bf 85} (2016), no. 298, 825--847.
		
	\bibitem{Dyson} F. Dyson, {\it A new symmetry of partitions}, J. Combinatorial Theory {\bf 7} (1969), 56--61.
	
	\bibitem{Garvan} Garvan, F., {\it Generalizations of {D}yson's rank and non-{R}ogers-{R}amanujan
		partitions}, Manuscripta Math., {\bf 84} (1994), n0. 3-4, 343--359.
	
	\bibitem{GORZ} M. Griffin, K. Ono, Ken, L. Rolen, and D. Zagier, {\it Jensen polynomials for the {R}iemann zeta function and other
			sequences}, Proc. Natl. Acad. Sci. USA, {\bf 116} (2019), no. 23, 11103--11110.
		
	
\bibitem{Gupta} H. Gupta, {\it Finite differences of the partition function}, Math. Comp. {\bf 32} (1978), 144 1241--1243.
	
	\bibitem{HardyRamanujan} G. Hardy and S. Ramanujan, \emph{Asymptotic formulae in combinatory analysis},
	Proc. London Math. Soc. Ser. 2 \textbf{17} (1918), 75--115.
	
	\bibitem{IJT} J. Iskander, V. Jain, and V. Talvola, {\it Exact formulae for the fractional partition functions}, Res. Number Theory {\bf 6} (2020), 201--215.
	
	\bibitem{Knessl2} C. Knessl, \textit{Asymptotic behavior of high-order differences of the plane partition function}, Discrete Math. 126 (1994), no. 1-3, 179?193
	
	\bibitem{Knessl1} C. Knessl and J. B. Keller, \textit{Asymptotic behavior of high-order differences of the partition
			function}, Comm. Pure Appl. Math., {\bf 44} (1991), 8-9, 1033--1045.
	 
	
\bibitem{LW} H. Larson and I. Wagner, {\it Hyperbolicity of the partition {J}ensen polynomials}, Res. Number Theory, {\bf 5} (2019), no. 2, Paper No. 19, 12.
	
	
	\bibitem{L1} D. Lehmer, {\it On the remainders and convergence of the series for the
			partition function}, Trans. Amer. Math. Soc., {\bf 46} (1939), 362--373.
		
		

\bibitem{L2} D. Lehmer, {\it On the series for the partition function}, Trans. Amer. Math. Soc., {\bf 43} (1938), no. 2, 271--295.
	

	
	\bibitem{LZ} Z. Liu and N. H. Zhou, \textit{Uniform asymptotic formulas for the Fourier coefficients of the inverse of theta functions}, \url{https://arxiv.org/abs/1903.00835}, preprint.
	
\bibitem{Merca} M. Merca and J. Katriel, \textit{A general method for proving the non-trivial linear
			homogeneous partition inequalities}, Ramanujan J., {\bf 51} (2020), no. 2, 245--266.
	
	
	\bibitem{Od} A. M. Odlyzko, \textit{Differences of the partition function}, Acta Arith., {\bf 49} (1988), no. 3, 237--254.
	
	
	\bibitem{RademacherExact} H. Rademacher, \emph{A convergent series for the partition function $p(n)$}, PNAS February 1, 1937 {\bf 23} (2), 78--84.
	
	\bibitem{Sch} R. Schneider, \textit{Nuclear partitions and a formula for $p(n)$},  \url{https://arxiv.org/abs/1912.00575}, preprint.
		
	
	
	
	\bibitem{Watson} G. Watson, {\it A treatise on the theory of Bessel functions}, Cambridge University Press (1995).
\end{thebibliography}
\end{document}